\documentclass[12pt]{amsart}

\usepackage{lmodern}
\usepackage[T1]{fontenc}
\usepackage[utf8]{inputenc}

\usepackage[a4paper,top=3.3cm,bottom=3cm,left=3cm,right=3cm,bindingoffset=5mm]{geometry}

\usepackage{amssymb}
\usepackage{paralist} 

\usepackage{thmtools}

\usepackage[english]{babel} 
\usepackage{csquotes} 
\usepackage[backend=bibtex, firstinits=true, style=alphabetic, doi=true, sorting=nyt, maxnames=10]{biblatex}
\usepackage{tikz}
\tikzstyle{normal}=[circle,inner sep=0pt, fill=black,  minimum size=1.5mm, draw]

\usepackage[colorlinks=true]{hyperref}
\usepackage[capitalize]{cleveref}

\declaretheorem[numberwithin=section]{theorem}
\declaretheorem[numberlike=theorem]{lemma}
\declaretheorem[numberlike=theorem]{proposition}
\declaretheorem[numberlike=theorem, name=Claim]{assertion}

\theoremstyle{definition}

\declaretheorem[numberlike=theorem]{question}

\theoremstyle{remark}
\declaretheorem[numberlike=theorem]{remark}

\newcommand{\ab}{\mathbf{a}}
\newcommand{\ub}{\mathbf{u}}

\newcommand{\af}{\mathfrak{a}}

\newcommand{\NN}{\mathbb{N}}
\newcommand{\ZZ}{\mathbb{Z}}
\newcommand{\QQ}{\mathbb{Q}}
\newcommand{\RR}{\mathbb{R}}
\newcommand{\kk}{\Bbbk}

\newcommand{\set}[1]{{\{#1\}}}
\newcommand{\with}{\,\colon\,}

\newcommand{\restr}[2]{{#1}|_{#2}}
\newcommand{\Hi}[2]{\tilde{H}_{#1}({#2};\kk)}

\newcommand{\Fc}{\mathcal{F}}

\DeclareMathOperator{\Tor}{Tor}
\DeclareMathOperator{\Char}{char}
\DeclareMathOperator{\reg}{reg}
\DeclareMathOperator{\pdim}{pdim}
\DeclareMathOperator{\lcm}{lcm}

\DeclareMathOperator{\DEG}{\mathbf{deg}}

\newcommand{\MA}{\mathcal{Z}_{\Delta}}

\newcommand{\gena}[1]{\mathbf{e}_{#1}}
\newcommand{\gen}[1]{\gena{\set{#1}}}

\newcommand{\genaE}[1]{\mathbf{\widehat{e}}_{#1}}
\newcommand{\genE}[1]{\genaE{\set{#1}}}
\newcommand{\cls}[1]{{[{#1}]}}

\let\phi\varphi

\newcommand{\abSc}{ab\#c}
\newcommand{\bcSa}{bc\#a}
\newcommand{\caSb}{ca\#b}

\begin{document}

\title[A non-Golod ring]{A non-Golod ring with a trivial product on its Koszul homology}

\author{Lukas Katth\"an}

\address{Goethe-Universit\"at Frankfurt, FB Informatik und Mathematik, 60054 Frankfurt am Main, Germany}
\email{katthaen@math.uni-frankfurt.de}


\begin{abstract}
	We present a monomial ideal $\af \subset S$ such that $S/\af$ is not Golod, even though the product in its Koszul homology is trivial.
	This constitutes a counterexample to a well-known result by Berglund and Jöllenbeck (the error can be traced to a mistake in an earlier article by Jöllenbeck).
	
	On the positive side, we show that if $R$ is a monomial ring such that the $r$-ary Massey product vanishes for all $r \leq \max(2, \reg R-2)$, then $R$ is Golod.
	In particular, if $R$ is the Stanley-Reisner ring of a simplicial complex of dimension at most $3$, then $R$ is Golod if and only if the product in its Koszul homology is trivial.
	
	Moreover, we show that if $\Delta$ is a triangulation of a $\kk$-orientable manifold whose Stanley-Reisner ring is Golod, then $\Delta$ is $2$-neighborly.
	This extends a recent result of Iriye and Kishimoto.
\end{abstract}

\maketitle

\section{Introduction}
Let $S = \kk[x_1, \dots, x_n]$ be a polynomial ring over some field $\kk$, endowed with the fine $\ZZ^n$-grading, and let $\af \subset S$ be a monomial ideal.
The (multigraded) \emph{Betti-Poincar\'e} series of $R := S/\af$ is the formal power series
\[ P_\kk^R(t,z_1,\dotsc,z_n) := \sum_{j \geq 0} \sum_{\ab \in\NN^n} \dim_\kk \Tor^R_j(\kk,\kk)_\ab t^j z_1^{a_1}\dotsm z_n^{a_n}, \]
where $\Tor^R_j(\kk,\kk)_\ab$ denotes the homogeneous component of $\Tor^R_j(\kk,\kk)$ in multidegree $\ab$.
We further consider the formal power series
\[
	Q_\kk^R(t,z_1,\dotsc,z_n) := \frac{\prod_{i=1}^n(1+t z_i)}{1-\sum_{j\geq 1}\sum_{\ab \in\NN^n} \dim_\kk H_j(K_R)_\ab t^{j+1} z_1^{a_1}\dotsm z_n^{a_n}},
\]
where $K_R$ denotes the Koszul complex of $R$.
The ring $R$ is called a \emph{Golod ring} if 
\begin{equation}\label{eq:ps}
	P_\kk^R(t,1,\dotsc,1) = Q_\kk^R(t,1,\dotsc,1)
\end{equation}
As reported by Golod in \cite{golod}, Serre proved that every ring $R$ satisfies the coefficientwise inequality $P_\kk^R(t,z_1,\dotsc,z_n) \leq Q_\kk^R(t,z_1,\dotsc,z_n)$.
Therefore, the Golod property is equivalent to the seemingly stronger condition that \[P_\kk^R(t,z_1,\dotsc,z_n) = Q_\kk^R(t,z_1,\dotsc,z_n)\text{.}\]
In the same article, Golod showed that $R$ satisfies \eqref{eq:ps} if and only if the product and all higher Massey products on the Koszul homology $H_*(K_R)$ are trivial.
Here, we say that the product is \emph{trivial} if the product of every two elements of positive homological degrees is zero, and the higher Massey products are trivial if they are all defined and contain only zero.

The main contribution of the present article is an example of a monomial ideal $\af \subset S$ such that the product in $H_*(K_{S/\af})$ is trivial, but $S/\af$ is not Golod. The example is given in \cref{thm:main}. As far as we know, this is the first example of a non-Golod ring $R$ with a trivial product in $H_*(K_R)$.

Our example provides a counterexample to a claim by Berglund and J\"ollenbeck:
\begin{assertion}[Theorem 5.1, \cite{BJ07}]\label{thm:wrong}
	Let $\af \subset S$ be a monomial ideal and let $R := S/\af$.
	Then the following are equivalent:
	\begin{enumerate}
		\item $R$ is Golod.
		\item The product in the Koszul homology of $R$ is trivial.
	\end{enumerate}
\end{assertion}
We would like to point out that the claim in \cite{BJ07} fails because its proof builds on an incorrect result of \cite{Jo06}.
A list of special cases (both new and known) in which \cref{thm:wrong} does hold is collected in \cref{thm:classes}.

The Golod property of quotients by monomial ideals is related to certain topological features of \emph{moment-angle complexes}.
Indeed, if $\af$ is a squarefree monomial ideal, then it can be interpreted as the Stanley-Reisner ideal of a simplicial complex $\Delta$.
The moment-angle complex $\MA$ is a certain topological space associated to $\Delta$, which was introduced by Davis and Januszkiewicz~\cite{DJ}.
A prominent feature of this space is that there is an isomorphism of $\kk$-algebras
\[ H^*(\MA;\kk) \cong H_*(K_R), \]
cf.~\cite[Theorem 7.7]{BP}.
A lot of research has been devoted to study the relation between the structure of $H_*(K_R)$ and the topology of $\MA$, see for example 
\cite{DS,IK, IK2, GPTW12}.
In terms of moment-angle complexes, our example gives rise to a moment-angle complex which is not formal but has a trivial cup-product in its cohomology.

In view of our example, it seems natural to ask whether one can bound the arity of the Massey products one needs to consider.
Indeed, if $R$ is a monomial ring, then it is Golod if all $r$-ary Massey products vanish for all $r \leq \max(2,\reg R - 2)$, see \Cref{thm:reg} below. In particular, the Stanley-Reisner ring of simplicial complex of dimension at most $3$ is Golod if and only if the product in its Koszul homology of $\kk[\Delta]$ is trivial.

Recently, Iriye and Kishimoto~\cite[Theorem 1.3]{IK2} showed that the Stanley-Reisner ring $\kk[\Delta]$ of a triangulated $\kk$-orientable surface $\Delta$ is Golod if and only if $\Delta$ is $2$-neighborly, i.e. if any two vertices are connected by an edge.
The methods we use to prove \cref{thm:reg} also allow us to generalize one of the implications to $\kk$-orientable manifolds of arbitrary dimension, cf. \Cref{thm:mf}.

This article is structured as follows.
In \cref{sec:prelim} some definitions and basic facts concerning Golod rings are recalled. Also, we describe how the Taylor complex can be used to compute Massey products.
After that, we prove our main result in \cref{sec:example}.
In the following \cref{sec:reg}, we prove \cref{thm:reg} and \cref{thm:mf}.
In \cref{sec:how} we sketch the considerations that led to us to find our main example.
In the last \cref{sec:remark} some remarks and an open questions are added.

\section{Preliminaries}\label{sec:prelim}
In this section, we recall some facts about Massey products and Golod rings.
We refer the reader to \cite{AvraG} and \cite{Avra} for a comprehensive treatment of general Golod rings.
Also, we describe how the Taylor resolution of a monomial ring can be used to compute Massey products.

\subsection{Massey products of DGAs}
Let us recall the definition of the Massey products of a differential graded algebra (DGA) $A$.
The binary Massey product is just the usual product which is inherited from $A$.
Let us denote it by $\mu_2(a_1,a_2)$.
For $n \geq 3$, the $n$-ary Massey product is a partially defined set-valued function, which assigns to $n$ elements $a_1, \dotsc, a_n \in H_*(A)$ a set $\mu_n(a_1, \dotsc, a_n) \subset H_*(A)$.
It is defined if there exist elements $a_{ij} \in A$ for $1 \leq i \leq j \leq n$, such that $da_{ii} = 0$, $\cls{a_{ii}} = a_i$ and 
\[ da_{ij} = \sum_{v=i}^j \bar{a}_{iv}a_{vj}, \]
where $\cls{a_{ii}}$ denotes the homology class of $a_{ii}$ and $\bar{a} = (-1)^{|a|+1}$. 
Then $\cls{\sum_{v=1}^n \bar{a}_{iv}a_{vj}}$ is called an ($n$-ary) Massey product of $a_1, \dotsc, a_n$ and $\mu_n(a_1, \dotsc, a_n)$ is the set of all these elements.
Further, we say that $A$ satisfies $(B_r)$, if all $k$-ary Massey products are defined and contain only zero for all $k \leq r$.
We will use the following lemma.
\begin{lemma}\label{lemma:unique}
	Let $A$ be a DGA satisfying $(B_{r-1})$.
	Then $\mu_{r}(a_1, \dotsc, a_{r})$ is defined and contains a single element for all $a_1, \dotsc, a_{r} \in H_*(A)$.
\end{lemma}
This is a special case of \cite[Proposition 2.3]{May}, see also \cite[Lemma 20]{Kraines}.
In the second reference, the result is claimed only for classes of odd degree, but the proof also holds in general.

By the following result, the Massey products depend only on the \enquote{hotomopy type} of a DGA.
It is essentially a special case of \cite[Theorem 1.5]{May}.
\begin{proposition}\label{prop:DGAmap}
	Let $A,B$ be DGAs, $f: A \to B$ be a map of DGAs and $a_1, \dotsc, a_r \in H_*(A)$.
	\begin{enumerate}
		\item If $\mu_r(a_1, \dotsc, a_r)$ is defined then also $\mu_r(f_*(a_1), \dotsc, f_*(a_r))$ is defined, and it holds that
		\begin{equation}\label{eq:quism}
			f_*(\mu_r(a_1, \dotsc, a_r)) \subseteq \mu_r(f_*(a_1), \dotsc, f_*(a_r)).\tag{$*$}
		\end{equation}
		\item If $f$ is a quasi-isomorphism, then $\mu_r(a_1, \dotsc, a_r)$ is defined if and only if $\mu_r(f_*(a_1), \dotsc, f_*(a_r))$ is defined and equality holds in \cref{eq:quism}
	\end{enumerate}
\end{proposition}
The statement in \cite[Theorem 1.5]{May} corresponding to our item (2) does not include the \enquote{if} part, but the latter follows from the proof given there.

\subsection{Computing Massey products via the Taylor resolution}
Let $S = \kk[x_1,\dotsc, x_n]$ be a polynomial ring, $\af \subset S$ a monomial ideal and $R := S/\af$. We are interested in two DGAs associated with $R$: On the one hand, its Koszul complex $K_R$ and on the other hand its \emph{Taylor resolution} $T_\bullet$.
Both DGAs inherit an \enquote{internal} multigrading from $R$ in addition to the natural \enquote{homological} $\NN$-grading.
For a homogeneous element $a$, we denote by $|a|$ its homological degree, by $\DEG a$ its internal multidegree and by $\deg a$ it internal $\ZZ$-degree.

We recall the definition of the Taylor complex $T_\bullet$, see also \cite[p. 67]{millersturm}.
Let $G(\af)$ denote the minimal set of monomial generators of $\af$ and choose a total order $\prec$ on $G(\af)$.
The choice of the order affects only the signs in the computations.

Then $T_\bullet$ is the complex of free $S$-modules with basis $\set{\gena{I} \with I \subseteq G(\af)}$.
The basis elements are graded by $|\gena{I}| := \#I$ and $\DEG \gena{I} := \DEG m_I$, where $m_I := \lcm(m \with m \in I)$ for $I \subseteq G(\af)$.
Further, the differential is given by
\[ \partial \gena{I} = \sum_{m \in I} (-1)^{\sigma(m,I)}\frac{m_I}{m_{I\setminus\set{m}}} \gena{I\setminus\set{m}}, \]
where $\sigma(m,I) := \#\set{m' \in I, m' \prec m}$.
The complex $T_\bullet$ carries a DGA structure (cf. \cite{Gemeda} or \cite[\S 5]{Froberg}) with the multiplication given by
\[ \gena{I} \cdot \gena{J} = \begin{cases}
	(-1)^{\sigma(I,J)} \frac{m_I m_J}{m_{I\cup J}} \gena{I \cup J} & \text{if } I \cap J = \emptyset,\\
	0& \text{otherwise},
\end{cases} \]
where $\sigma(I,J) := \#\set{(m,m') \in I \times J \with m' \prec m}$.

The importance of the Taylor resolution for us stems from the fact that there are quasi-isomorphisms of DGAs \[T_\bullet \otimes_S \kk \leftarrow T_\bullet \otimes_S K_S \to R \otimes_S K_S = K_R,\]
which imply that $H_*(K_R) = \Tor_*^S(R, \kk) = H_*(T_\bullet \otimes_S \kk)$ and we can compute Massey products using either $K_R$ or $T_\bullet \otimes_S \kk$ (cf. \Cref{prop:DGAmap} (2)).
Finally, we note that Massey products are compatible with the grading in the following sense:
If $a_1, \dotsc, a_r \in H_*(K_R)$ are homogeneous elements and $a \in \mu_r(a_1, \dots, a_r)$ is a Massey product, then $\DEG a = \sum_i \DEG a_i$ and $|a| = \sum_i |a_i| + (r-2)$.

We now give a combinatorial description of the maps in $T_\bullet \otimes_S \kk$.
It is clear that $T_\bullet \otimes_S \kk$ is a direct sum of complexes of vector spaces, and the only multidegrees in which $T_\bullet \otimes_S \kk$ is non-trivial are the degrees in the lcm-lattice $L_\af := \set{\DEG m_I \with I \subseteq G(\af)}$ of $\af$.
For $\ub \in L_\af$ let $G_\ub := \set{m \in G(\af) \with \DEG m \leq \ub}$ and let
\[ \Fc_\ub := \set{ I \subseteq G_\ub \with \DEG \lcm(G_\ub \setminus I) = \ub}. \]
Note that $\Fc_\ub$ is a simplicial complex.
The next observation follows directly from the definition of the Taylor complex:
\begin{proposition}\label{prop:fiber}
	Let $\ub \in L_\af$. 
	Then the strand $(T_\bullet \otimes \kk)_\ub$ in multidegreee $\ub$ is isomorphic to the complex of simplicial cochains on $\Fc_m$, up to reversing and shifting the homological degrees.
	
	More precisely, if $I \subseteq G(\af)$ and $\DEG \lcm(I) = \ub$, then this isomorphism maps $\gena{I} \otimes_S 1_\kk$ to the
	$(\#G_\ub - \#I-1)$-cochain supported on the single simplex $G_\ub \setminus I$.
\end{proposition}

Now we turn to the computation of Massey products.
In order to simplify the notation, we define $\genaE{I} := \gena{I} \otimes_S 1_\kk \in T_\bullet \otimes \kk$ for $I \subseteq G(\af)$.
First, we give two simple sufficient conditions for the vanishing of products.
\begin{lemma}\label{lem:paar}
	Let $\af \subseteq S$ be a monomial ideal.
	\begin{enumerate}
	\item Let $a,b \in H_*(K_R)$ be two Koszul cycles which are homogeneous with respect to the multigrading.
	If the multidegrees of $a$ and $b$ are not orthogonal (i.e. they have a non-zero component in common), then $a \cdot b = 0$.
	\item Let $a,b \in G(\af)$ be coprime generators of $\af$.
	Then $\cls{\genE{a} \cdot \genE{b}} \in H_*(T_\bullet \otimes_S \kk)$ is zero if and only if there exists a generator $c \neq a, b$ of $\af$ which divides the least common multiple of $a$ and $b$. 
	\end{enumerate}
\end{lemma}
\begin{proof}
	Both claims are invariant under polarization, so we assume that $\af$ is squarefree.
	In this case the first claim is clear, because $a \cdot b$ would have a non-squarefree multidegree.
	
	So consider the second claim. 
	Assume there exist a generator $c \in G(\af)$ which divides $\lcm(a,b)$ and choose an order on $G(\af)$ such that $a \prec b \prec c$. Then 
	\[ \partial \gen{a,b,c} = \frac{a}{\gcd(a,c)} \gen{b,c} - \frac{b}{\gcd(b,c)} \gen{a,c} + \gen{a,b}, \]
	so $\genE{a} \cdot \genE{b} = \genE{a,b} = \partial \genE{a,b,c}$ is a boundary in $T_\bullet \otimes_S \kk$.
	On the other hand, if there does not exist such a $c$, then $\genE{a,b}$ is the only generator of $T_\bullet \otimes_S \kk$ in the multidegree $\DEG\lcm(a,b)$, and thus it cannot be a boundary.
\end{proof}

Our next Lemma gives a combinatorial description of ternary Massey products of elements in homological degree one.
\begin{lemma}\label{lem:3mas}
	Let $a,b,c \in G(\af)$ and assume that $a \prec b \prec c$.
	\begin{enumerate}
		\item If $\mu_3(\cls{\genE{a}}, \cls{\genE{b}}, \cls{\genE{c}})$ is defined and contains not only zero, then $a,b$ and $c$ are pairwise coprime and there exist elements $ab, bc \in G(\af) \setminus\set{a,b,c}$ such that $ab \mid \lcm(a,b)$ and $bc \mid \lcm(b,c)$.
		\item If these conditions are satisfied, then $\mu_3(\cls{\genE{a}}, \cls{\genE{b}}, \cls{\genE{c}})$ is defined and contains the class $\cls{-\genE{a,ab,b,c} - \genE{a,b,bc,c}}$. 
		
		In particular, if $\ub := \DEG\lcm(a,b,c)$, then $\mu_3(\cls{\genE{a}}, \cls{\genE{b}}, \cls{\genE{c}})$ is non-trivial if the simplicial cochain
		\[ e^*_{G_\ub \setminus\set{a,ab,b,c}} + e^*_{G_\ub \setminus\set{a,b,bc,c}} \]
		is not a coboundary in $\Fc_\ub$.
	\end{enumerate}
\end{lemma}
\begin{proof}
	Again, we may assume that $\af$ is squarefree.
	\begin{asparaenum}
		\item If $a,b$ and $c$ are not pairwise coprime, then the elements of the ternary Massey product
		$\mu_3(\cls{\genE{a}}, \cls{\genE{b}}, \cls{\genE{c}})$ have a non-squarefree multidegree and hence the Massey product contains only zero, contradicting our assumption.
		Moreover, if there does not exist an $ab \in G(\af)$ with the claimed properties, then the strand of $T_\bullet \otimes_S \kk$ in the multidegree $\DEG \genE{a,b}$ contains only the vector space spanned by $\genE{a,b} = \genE{a} \cdot \genE{b}$, so this is not a boundary. 
		Hence $\mu_3(\cls{\genE{a}}, \cls{\genE{b}}, \cls{\genE{c}})$  is not defined, contradicting the hypothesis.
		By the same argument, the element $bc$ exists.
		
		\item Assume that $a \prec ab \prec b \prec bc \prec c$.
		By the preceding lemma, we have that $\partial\genE{a,ab,b} = - \genE{a,b}$ and $\partial\genE{b,bc,c} = - \genE{b,c}$. Hence the class of $- \genE{a}\cdot\genE{b,bc,c} - \genE{a,ab,b}\cdot\genE{c} = - \genE{a,ab,b,c} - \genE{a,b,bc,c}$ is a Massey product of $\cls{\genE{a}}, \cls{\genE{b}}$ and $\cls{\genE{c}}$.
		The last statement is immediate from \Cref{prop:fiber}.
	\end{asparaenum}
\end{proof}

\section{Discussion of the example} \label{sec:example}
\begin{figure}[t]
\begin{tikzpicture}[scale=1.5, every node/.style=normal]
	\begin{scope}
		\fill[gray!30] (1,0)-- (60:1) -- (120:1) -- (180:1) -- (240:1) -- (300:1) -- cycle;
		
		\path (0,0)
			node[label=above:a] (a) at (60:1) {}
			node[label=right:ab] (ac) at (0:1) {}
			node[label=below:b] (c) at (300:1) {}
			node[label=below:bc] (bc) at (240:1) {}
			node[label=left:c] (b) at (180:1) {}
			node[label=above:ca] (ab) at (120:1) {};
		
		\path (0,0) node[label=below left:$\abSc$] (bac) at (0,0) {};
	
		\draw (ab) -- (b) -- (bc) -- (c) -- (ac) -- (a) -- (ab);
		\draw (b) -- (bac) -- (ac);
		\draw (a)  -- (bac) -- (c);
	\end{scope}
	\begin{scope}[xshift=3cm]
		\fill[gray!30] (1,0)-- (60:1) -- (120:1) -- (180:1) -- (240:1) -- (300:1) -- cycle;
		
		\path (0,0)
			node[label=above:a] (a) at (60:1) {}
			node[label=right:ab] (ac) at (0:1) {}
			node[label=below:b] (c) at (300:1) {}
			node[label=below:bc] (bc) at (240:1) {}
			node[label=left:c] (b) at (180:1) {}
			node[label=above:ca] (ab) at (120:1) {};
		
		\path (0,0) node[label=left:$\bcSa$] (abc) at (0,0) {};
	
		\draw (ab) -- (b) -- (bc) -- (c) -- (ac) -- (a) -- (ab);
		\draw (ab) -- (abc) -- (bc);
		\draw (a)  -- (abc) -- (c);
	\end{scope}
\end{tikzpicture}
\caption{The cellular resolution of $S/\af$. This is a $3$-ball with one $3$-cell and the figure shows the top and the bottom of this ball.}
\label{fig:cellres}
\end{figure}
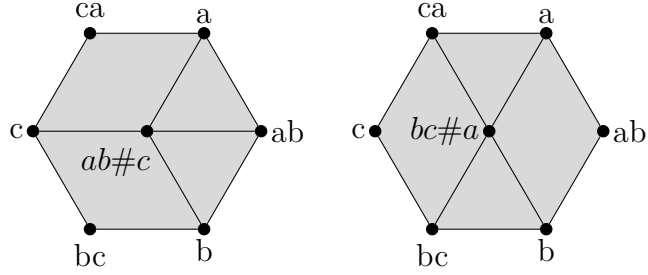
\newcommand{\ga}[1]{m_{#1}}
In this section, we prove our main result:
\begin{theorem}\label{thm:main}
	Let $\kk$ be a field, $S = \kk[x_1,x_2,y_1,y_2,z]$ and let $\af \subset S$ be the ideal with the following generators:
	\begin{align*}
		\ga{a} &:= x_1x_2^2 & \ga{ab} &:= x_1x_2y_1y_2  & \ga{\abSc} &:= x_1y_1z\\
		\ga{b} &:= y_1y_2^2 & \ga{bc} &:= y_2^2z^2 		& \ga{\bcSa} &:= x_2^2y_2^2z\\
		\ga{c} &:= z^3 	 & \ga{ca} &:= x_2^2z^2
	\end{align*}
	Then the product in $H_*(K_{S/\af}) = \Tor^S_*(S/\af, \kk)$ is trivial, but $S/\af$ is not Golod.
	More precisely, the ternary Massey product $\mu_3(\cls{\genE{\ga{a}}}, \cls{\genE{\ga{b}}}, \cls{\genE{\ga{c}}})$ is nonzero.
\end{theorem}
	In the first version of this article a slightly more complicated example was presented.
	After that version appeared on ArXiv, Roos pointed out that that example could be simplified to obtain the ideal given above.
	The names of the generators are chosen in accordance with the discussion in the proof of \Cref{thm:min}.

\renewcommand{\ga}[1]{{#1}}

\begin{proof}
To simplify the notation, we write $a,b,ab,\dots$ instead of $m_a, m_b, m_{ab}, \dotsc$. In particular, $ab$ does \emph{not} denote the product of $a$ and $b$, but the generator $m_{ab}$.

Our first goal is to describe the structure of the minimal free resolution of $S/\af$.
For this, consider the cell complex $X$ depicted in \cref{fig:cellres}.
It is a labeled cell complex in the sense of \cite{BS98}, and we claim that the corresponding complex $F_X$ of free $S$-modules is a minimal free resolution of $S/\af$. 
It is clear that its zeroth homology is $S/\af$, so we need to show that its higher homology groups vanish.
By \cite[Proposition 1.2]{BS98}, this holds if and only if $X_{\leq \ab}$ is acyclic over $\kk$, where for $\ab\in\NN^n$, the subcomplex $X_{\leq \ab} \subseteq X$ contains exactly those faces whose labels are coordinatewise less or equal to $\ab$.
Now, $X$ can be embedded into $\RR^3$ and hence none of its induced subcomplexes has torsion in its homology groups (cf. \cite[Corollary 3.45]{hat}).
Therefore, if $F_X$ is acyclic over at least one field, then it is acyclic over any field.

To show the former, we consider the case $\kk=\QQ$.
Over $\QQ$, one can compute a minimal free resolution of $S/\af$ using 
\texttt{Macaulay2} \cite{M2}, and then verify by inspection that this resolution is indeed isomorphic to $F_X$.
Moreover, all entries in the matrices of $F_X$ are non-constant monomials, and hence $F_X$ is minimal.

Thus, $S/\af$ has a minimal cellular resolution supported on $X$ over all fields.
In particular, its Betti diagram does not depend on the field, and it is given by the following:
\begin{center}
	\begin{tabular}{r|rrrrr}
		& 0 & 1 & 2 & 3 & 4 \\\hline
		0& 1 & . & . & . & . \\
		1& . & . & . & . & . \\ 
		2& . & 4 & . & . & . \\
		3& . & 3 &10 & 2 & . \\
		4& . & 1 & 4 & 6 & . \\
		5& . & . & . & . & 1 \\
	\end{tabular}
\end{center}

From
this
one can read off that, for degree reasons, the only products which can possibly be non-zero are the following:
\begin{enumerate}
\item $\genE{i} \cdot \genE{j}$ for $i,j \in \set{\ga{a},\ga{b},\ga{c}}, i \neq j$,
\item $\genE{i} \cdot \genE{\ga{\abSc}}$ for $i \in \set{\ga{a},\ga{b},\ga{c}}$, and
\item $\genE{i} \cdot h$ for $i \in \set{\ga{a},\ga{b},\ga{c}, \ga{\abSc}}$ and $h$ is a generator with $|h| = 3$ and $\deg h = 6$, i.e. the \enquote{2} in the Betti diagram. 
\end{enumerate} 
It holds that $\ga{ij} | \lcm(\ga{i}, \ga{j})$ for $(i,j) \in \set{(a,b),(b,c),(c,a)}$, hence these products are zero by \cref{lem:paar}.
Moreover, $\ga{\abSc}$ is not coprime with $\ga{a}, \ga{b}$ or $\ga{c}$, so the second set of products are zero for degree reasons.
Finally, it is not difficult to see that the generators $h$ with $|h| = 3$ and $\deg h = 6$ correspond to the triangles $\set{a,ab,\abSc}$ and $\set{b,ab,\abSc}$ in \cref{fig:cellres}. 
So their multidegrees are not orthogonal to the multidegree of any $\genE{i}$ for  $i \in \set{\ga{a},\ga{b},\ga{c}, \ga{\abSc}}$, and hence this products are zero as well.
Thus, all products on $\Tor^S_*(S/\af, \kk)$ vanish.
Alternatively, one can use the \texttt{Macaulay2} command \texttt{isHomologyAlgebraTrivial} from the package \texttt{DGAlgebras} by Frank Moore (which is distributed along with \texttt{Macaulay2}) to verify this.

Next, we show that the ternary Massey product $\mu_3(\cls{\genE{\ga{a}}},\cls{\genE{\ga{b}}},\cls{\genE{\ga{c}}})$ is non-zero.
It holds that $\ga{ab} | \lcm(\ga{a},\ga{b})$ and $\ga{bc} | \lcm(\ga{b},\ga{c})$, so by \cref{lem:3mas} we need to show that the simplicial cochain 
\[
	\omega :=
	e^*_{G_\ub \setminus\set{\ga{a},\ga{ab},\ga{b},\ga{c}}} + e^*_{G_\ub \setminus\set{\ga{a},\ga{b},\ga{bc},\ga{c}}} 
\]
is not a coboundary in $\Fc_\ub$, where $\ub := \DEG\lcm(\ga{a},\ga{b},\ga{c})$.
Note that $G_\ub = G(\af)$.
To obtain an explicit description of $\Fc_\ub$, we note that a set $I \subseteq G(\af)$ is contained in $\Fc_\ub$ if and only if for each variable, $I$ contains a generator having the maximal degree in this variable.
By using this description of $\Fc_\ub$, we find
that $\Fc_\ub$ is the simplicial complex with vertex set $G(\af)$ and minimal non-faces $\set{\ga{b}}\footnote{
Note that $b$ is not a vertex of $\Fc_\ub$, even though it is contained in the vertex set. Some authors call such an element a \enquote{ghost} vertex.}
, \set{\ga{a},\ga{ab},\ga{\bcSa}}, \set{\ga{a},\ga{ac},\ga{\caSb}}, \set{\ga{c},\ga{bc},\ga{\bcSa}}$ and $\set{\ga{c},\ga{ac},\ga{\caSb}}$.

The verification that $\omega$ is not a coboundary can be done with any software system capable of computing simplicial cohomology, for example using the package \texttt{simpcomp} \cite{simpcomp} for the \texttt{GAP} system \cite{GAP}.
Moreover, the Betti numbers of $\Fc_\ub$ equal the multigraded Betti numbers of $S/\af$ in degree $\ub$ and thus they do not depend on the field $\kk$. Hence, the cohomology of $\Fc_\ub$ has no torsion and so our claim holds independently of the field.
\end{proof}

\begin{remark}
	The Poincar\'e-Betti series of our example can be computed using the formula given in \cite[Theorem 1]{B06}.
	Its specialization to the $\ZZ$-grading is given by
	\[
		P_\kk^R(t,z, \dotsc, z) = \frac{(1+ zt)^5}{1 - t((4 z^3 + 3 z^4 + z^5)t + (10z^5 + 4 z^6)t^2 +  (2z^6 + 6 z^7) t^3 - z^9 t^5)}.
	\]
	Evaluating both sides of \cref{eq:ps} using the preceding expression and the Betti table of $S/\af$ yields the following:
	\begin{align*}
	P_\kk^R(t,1,\dotsc,1) &= 1 + 5t + 18 t^2 + 64 t^3 + 227 t^4 + {805} t^5 + \dotsb\\
	Q_\kk^R(t,1,\dotsc,1) &= 1 + 5t + 18 t^2 + 64 t^3 + 227 t^4 + {806} t^5 + \dotsb
	\end{align*}
	These series differ, reflecting the fact that $S/\af$ is not Golod.
\end{remark}

\begin{remark}
		The polarization of the ideal in \cref{thm:main} is the Stanley-Reisner ideal of some simplicial complex $\Delta$ of dimension $5$.
		By taking its $4$-skeleton, one obtains an example of a $4$-dimensional simplicial complex $\Gamma$, such that $\kk[\Gamma]$ is not Golod but has a trivial product in its Koszul homology.
		Indeed, the product stays trivial under taking the skeleton by \cite[Corollary 5.1]{Golod1}, and the non-vanishing Massey product is a $4$-cycle, so it cannot become a boundary when we remove simplices of higher dimension from $\Delta$.
		In particular, this shows that part (7) of \cref{thm:classes} below is best possible.
\end{remark}

\section{A bound in terms of regularity}\label{sec:reg}
In this section we show a weaker version of \cref{thm:wrong}:

\begin{theorem}\label{thm:reg}
	Let $\af \subset S$ be a monomial ideal, let $R := S/\af$ and let $r := \max(2, \reg R - 2)$.
	If $K_R$ satisfies $(B_{r})$, then the ring is Golod.
\end{theorem}
Recall that $K_R$ satisfies $(B_r)$ if all $k$-ary Massey products are defined and contain only zero for all $k \leq r$.
We prove \cref{thm:reg} at the end of this section.
First, we need to introduce some notation.
Let $\Delta$ be a simplicial complex with vertex set $V$.
The Stanley-Reisner ideal of $\Delta$ is the ideal
\[ I_\Delta = \left(\prod_{v \in M} X_v \with M \subseteq V, M \notin \Delta \right) \]
in the polynomial ring $S = \kk[X_v \with v \in V]$.
The Stanley-Reisner ring of $\Delta$ is then the quotient $\kk[\Delta] = S/I_\Delta$, cf. \cite[Chapter 5]{BH}.
For a subset $U \subseteq V$, we denote by 
$\restr{\Delta}{U} := \set{F \in \Delta \with F \subseteq U}$
the restriction of $\Delta$ to $U$.
\begin{lemma}\label{lemma:restrict}
	Let $\Delta$ be a simplicial complex such that $K_{\kk[\Delta]}$ satisfies $(B_r)$.
	Then the same holds for every restriction of $\Delta$.
\end{lemma}
\begin{proof}
	Let $U \subseteq V$ be a subset.
	The Stanley-Reisner ring of $\Delta|_U$ is an algebra retract of $\kk[\Delta]$, i.e. there are maps of algebras $\iota: \kk[\Delta|_U] \to \kk[\Delta]$ and $p: \kk[\Delta] \to \kk[\Delta|_U]$ such that $p \circ \iota = \mathrm{id}$.
	These maps induce maps on the Koszul complexes $K_\iota: K_{\kk[\Delta|_U]} \to K_{\kk[\Delta]}$ and $K_p: K_{\kk[\Delta]} \to K_{\kk[\Delta|_U]}$, 
	which are maps of DGAs and satisfy $K_p \circ K_\iota = \mathrm{id}$.
	
	Let $a_1, \dotsc, a_k \in H_*(K_{\kk[\Delta|_U]})$ for some $k \leq r$.
	By our assumption that $K_{\kk[\Delta]}$ satisfies $(B_r)$, $\mu_k((K_\iota)_*(a_1), \dots (K_\iota)_*(a_k))$ is defined and contains only zero.
	Hence \Cref{prop:DGAmap} implies that
	\[\mu_k((K_p \circ K_\iota)_*(a_1), \dots (K_p \circ K_\iota)_*(a_k)) = \mu_r(a_1, \dotsc, a_k) \]
	is defined as well and contains zero.
	Moreover, it contains only zero, because $(K_\iota)_*$ is an injective map from $\mu_r(a_1, \dotsc, a_k)$ to $\mu_k((K_\iota)_*(a_1), \dots (K_\iota)_*(a_k)) = \set{0}$.
\end{proof}

The next lemma is the key step in our proof of \cref{thm:reg}.
\begin{lemma}\label{lem:minimal}
	Let $\Delta$ be a simplicial complex.
	Assume that there exists a nonzero cohomology class $\alpha \in H^d(\Delta;\kk)$ for some $d \geq 1$, such that the restriction of $\alpha$ to any induced subcomplex of $\Delta$ is zero.
	
	If $\Delta$ is not $2$-neighborly, then $K_{\kk[\Delta]}$ does not satisfy $(B_2)$, i.e. there exists a non-vanishing product in the Koszul homology.
\end{lemma}
Here, \enquote{$2$-neighborly} means that every two vertices are connected by an edge.
\begin{proof}
	Let $V$ be the vertex set of $\Delta$.
	For $i \in \NN$ and two non-empty disjoint subsets $I, J \subset V$,
	we write $\phi_i^{I,J}\colon \Hi{i}{\restr{\Delta}{I \cup J}} \rightarrow \Hi{i}{\restr{\Delta}{I} * \restr{\Delta}{J}}$ for the map induced by the inclusion $\restr{\Delta}{I \cup J} \hookrightarrow \restr{\Delta}{I} * \restr{\Delta}{J}$.
	Recall that these maps vanish if and only if all products on the Koszul homology vanish \cite[Proposition 6.3]{IK}.
	For a chain $\tau = \sum_{\sigma \in \Delta} c_\sigma \sigma$ on $\Delta$ we define
	\[ V(\tau) := \set{v \in V \with \exists \sigma \in \Delta: v \in \sigma, c_\sigma \neq 0}. \]
	Our hypothesis implies that
	\begin{equation}\label{eq:min}
		\langle \omega,\alpha \rangle = 0 \text{ for all } \omega \in H_d(\Delta) \text{ with } V(\omega) \neq V.
	\end{equation}
	If $\Delta$ is not $2$-neighborly, then there exist two vertices $v,w$ which are not connected by an edge.
	Set $I := \set{v,w}$ and $J := V \setminus I$.
	Note that $d \geq 1$ implies that $\Delta$ has more than two vertices and hence $J \neq \emptyset$.
	As $\alpha \neq 0$, there exists an $\omega\in H_d(\Delta)$ such that $\langle \omega,\alpha \rangle \neq 0$.
	We claim that $\phi_d^{I,J}(\omega) \neq 0$, and hence there exists a non-vanishing product in the Koszul homology.
	
	Before we prove the claim, we define some auxiliary maps.
	For a chain $\tau = \sum_{\sigma} c_\sigma \sigma$ on $\Delta$ we set
	\begin{align*}
		\tau_v &:= \sum \set{c_\sigma \sigma \with v \in \sigma} \quad\text{and}\\
		\tau^v &:= \sum \set{c_\sigma (\sigma \setminus \set{v}) \with v \in \sigma}.
	\end{align*}

	Choose a linear order on $V$ such that $v$ is the smallest vertex and orient $\Delta$ accordingly.
	Under this convention it holds that $d \tau_v = (d \tau)_v + \tau^v$.
	
	Now we turn to the proof of our claim.
	Assume the contrary, i.e. that $\phi_d^{I,J}(\omega) = 0$.
	Then there exists a $(d+1)$-chain $\tau$ in $\restr{\Delta}{I}*\restr{\Delta}{J}$ such that $d \tau = \omega$.
	Set $\omega' := d \tau_v = \omega_v +  \tau^v$.
	Note that $\omega'$ is in fact a chain on $\Delta$, because both $\omega_v$ and $\tau^v$ are chains on $\Delta$.
	(For the latter, recall that $\tau  \in \restr{\Delta}{I}*\restr{\Delta}{J}$ and $\restr{\Delta}{I}$ are just the two disconnected vertices $v$ and $w$.)
	Further, $v \notin V(\omega - \omega')$, hence by \cref{eq:min} it follows that $\langle \omega-\omega', \alpha\rangle = 0$ and thus
	\[ \langle \omega', \alpha\rangle = \langle\omega, \alpha\rangle \neq 0.\]
	
	As $v$ and $w$ are not connected in $\restr{\Delta}{I}*\restr{\Delta}{J}$, it holds that $w \notin V(\tau_v)$.
	Hence $w \notin V(\omega')$ as well, contradicting \cref{eq:min}.
\end{proof}

From the preceding lemma, we obtain the following result, which we consider to be of independent interest:
\begin{proposition}\label{thm:zero}
	Let $\af \subset S$ be a monomial ideal, let $R := S/\af$ and assume that $K_R$ satisfies $(B_{r-1})$ for some $r \geq 3$.
	Let $a_1, \dots, a_r \in H_*(K_R)$ be elements of the Koszul homology.
	If $\deg a_i = |a_i| + 1$ for some $i$ (i.e. $a_i$ lies in the $2$-linear strand of $K_R$), then $\mu_r(a_1,\dotsc, a_r) =\set{0}$.
\end{proposition}
\begin{proof}
	The claim is invariant under polarization, so we may assume that $\af$ is squarefree and thus is the Stanley-Reisner ideal of some simplicial complex $\Delta$ with vertex set $V$.
	Assume to the contrary that $\mu_r(a_1,\dotsc, a_r)$ contains a non-zero element $\alpha$.
	
	For any subset $U \subseteq V$, consider the map $K_{p,U}: K_{\kk[\Delta]} \to K_{\kk[\restr{\Delta}{U}]}$ of DGAs from the proof of \cref{lemma:restrict}.
	By replacing $V$ with a suitable subset, we may assume that $H_*(K_{p,U})(\alpha) = 0$ for every proper subset $U \subsetneq V$.
	Note that $K_{\kk[\restr{\Delta}{U}]}$ satisfies $(B_{r-1})$ because $K_{\kk[\Delta]}$ does (cf. \cref{lemma:restrict}), and $\alpha$ is also a Massey product of $\restr{\Delta}{U}$.
	So we may replace $\Delta$ by $\restr{\Delta}{U}$. 

	Assume that $\deg a_i = |a_i| + 1$ for some $i$. 
	The element $a_i$ is nonzero, because otherwise zero would be contained in $\mu_r(a_1,\dotsc, a_r)$.
	But by \cref{lemma:unique}, this set contains only one element $\alpha$ and we assumed that to be non-zero.
	The element $a_i$ corresponds to a $0$-class in some restriction of $\Delta$, and as it is non-zero, this restriction is disconnected.
	But now $\Delta$ satisfies the hypothesis of \cref{lem:minimal} and it is not $2$-neighborly, thus $K_{\kk[\Delta]}$ does not satisfy $(B_2)$, a contradiction.
\end{proof}

\noindent \Cref{thm:reg} follows from \cref{thm:zero} by degree considerations:
\begin{proof}[Proof of \cref{thm:reg}]
	Assume that $K_R$ satisfies $(B_{r-1})$ but not $(B_r)$ for some $r \geq 3$.
	Then there exist homogeneous elements $a_1, \dots, a_r \in H_{*}(K_{R})$ such that $\alpha \in \mu_r(a_1, \dots, a_r)$ is nonzero.
	Now \cref{thm:zero} implies that $\deg a_i \geq |a_i| + 2$ for all $i$.
	Hence 
	\[ \deg \alpha = \sum_i \deg a_i \geq  \sum_i |a_i| + 2r = |m| +r + 2 \]
	Thus, $\alpha \neq 0$ implies that $r \leq \reg_S R - 2$.
\end{proof}

We close this section with another consequence of \cref{lem:minimal}.
It extends the result on Golod surfaces by Iriye and Kishimoto \cite[Theorem 1.3]{IK2}.
\begin{theorem}\label{thm:mf}
	Let $\Delta$ be a triangulation of a $\kk$-orientable manifold.
	If $\kk[\Delta]$ is Golod, then $\Delta$ is $2$-neighborly.
\end{theorem}

\begin{proof}
	The fundamental class of $\Delta$ satisfies the hypothesis of \cref{lem:minimal}.
\end{proof}
The converse of this result holds if $\Delta$ is two-dimensional, cf.~\cite[Theorem 1.3]{IK2}.
In higher dimensions, the converse does not hold.
Indeed, the boundary complex of any simplicial $2$-neighborly polytope is Gorenstein* and thus not Golod.
To see the latter, note that the Kozyul homology of a Gorenstein ring is a Poincar\'e algebra and thus cannot have only zero products, cf. \cite[Theorem~3.4.5]{BH}.

Moreover, the assumption that $\Delta$ is $\kk$-orientable cannot be removed.
For example let $\Delta$ be the complex obtained from the usual $6$-vertex triangulation of the real projective plane by applying a stellar subdivision to one of the facets.
Then $\Delta$ is not $2$-neighborly but $\kk[\Delta]$ is Golod if (and only if) $\Char \kk \neq 2$. See \cite[Example 4.2]{Golod1} for a detailed discussion of this example.

\section{How the example was found}\label{sec:how}
In this section we sketch the considerations that lead to \cref{thm:main}.
Our approach is to try to obtain necessary combinatorial conditions on the generators of an ideal with the desired properties.
We illustrate the techniques employed in the proof of the following result.

\begin{proposition}\label{thm:min}
	Let $\af \subset S = \kk[x_1,\dotsc, x_n]$ be a monomial ideal,
	such that the product in $H_*(K_{S/\af}) = \Tor^S_*(S/\af, \kk)$ is trivial, but $S/\af$ is not Golod.
	
	Then $n \geq 5$ and $\af$ has at least $8$ minimal generators.
\end{proposition}
We only provide a sketch of the proof, because some of the arguments are quite repetitive.
\begin{proof}[Sketch of the proof]
	We first show that $\af$ has at least $8$ generators.
	The proof goes by contradiction, so assume that $\af$ satisfies the assumption but has less than $8$ generators.
	As the claim is invariant under polarization, we may also assume that $\af$ is squarefree.
	
	Our first step is to show that we only need to consider ternary Massey products.
	As $S/\af$ is not Golod, there is a nonzero Massey product $\mu_r(a_1, \dotsc, a_r)$ for some $r \geq 3$ and $a_1, \dotsc, a_r \in H_*(K_{S/\af})$.
	All Betti numbers of $S/\af$ have squarefree multidegrees, so the multidegree of this Massey product is squarefree as well, and hence the multidegrees of the $a_i$ have disjoint supports.
	For each $a_i$ there exists at least one generator $m_i$ of $\af$ whose multidegree is componentwise smaller than the multidegree of $a_i$. Thus we obtain $r$ minimal generators of $\af$ which are pairwise coprime.
	We assumed that the product in $H_*(K_{S/\af})$ is trivial, so \cref{lem:paar} implies that there are $\binom{r}{2}$ generators $m_{ij}$ of $\af$ with $m_{ij} | \lcm(m_i, m_j)$ for $1 \leq i < j \leq r$.
	Now $r + \binom{r}{2} \leq \#G(\af) = 7$ yields that $r = 3$.
	
	From now on, we identify squarefree monomials with finite sets.
	If we want to explicitly turn a finite set $a$ into a monomial we write $m_a := \prod_{i \in a} x_i$.
	Moreover, we are computing in $T_\bullet \otimes_S \kk$, so all non-constant monomials are in fact zero. 
	
	Our considerations so far yields that $\af$ has at least six generators $a,b,c,ab,bc,ca$ satisfying that
	\begin{enumerate}
		\item $a,b,c$ are pairwise disjoint, and
		\item $ij \subset i \cup j$ and $ij \cap i, ij \cap j \neq \emptyset$ for $(i,j) \in \set{(a,b),(b,c),(c,a)}$.
	\end{enumerate}
	We choose a total order on them by setting $a \prec ab \prec b \prec bc \prec c \prec ca$.
	By our hypothesis, the product in $H_*(K_{S/\af})$ is trivial and $\mu_3(\cls{\genE{a}},\cls{\genE{b}},\cls{\genE{c}})$ is non-trivial. So by \cref{lemma:unique}, $\mu_3(\cls{\genE{a}},\cls{\genE{b}},\cls{\genE{c}})$ contains only one element, and by \cref{lem:3mas} this element is the class of $\omega := -\genE{a,ab,b,c} - \genE{a,b,bc,c}$.
	
	We are going to derive information about the generators of $\af$ from the assumption that $\omega$ is not a boundary.
	It holds that
	\[ \begin{split}\partial \genE{a,ab,b,bc,c} =
	 m_{a \setminus ab} \genE{ab,b,bc,c} - \genE{a,b,bc,c} + m_{b \setminus (ab \cup bc)} \genE{a,ab,bc,c} \\
	 - \genE{a,ab,b,c} + m_{c \setminus bc} \genE{a,ab,b,bc}.
	 \end{split} \]
	The first and the last term are zero, because $a \setminus ab \neq \emptyset$ and $c \setminus bc\neq\emptyset$.
	Hence
	\[ \omega = - m_{b \setminus (ab \cup bc)} \genE{a,ab,bc,c}.\]
	For this to be non-zero, it is necessary that $b \setminus (ab \cup bc) = \emptyset$, or equivalently
	\begin{equation}\label{eq:b}
	b \subset ab \cup bc
	\end{equation}
	Next, we compute that
	\[\begin{split} \partial \genE{a,ab,bc,c,ca} = 
	m_{a \setminus(ab\cup ca)} \genE{ab,bc,c,ca} - m_{ab \setminus(a\cup bc)} \genE{a,bc,c,ca} + m_{bc \setminus(ab\cup c)} \genE{a,ab,c,ca} \\
	 - m_{c \setminus(bc\cup ca)}\genE{a,ab,bc,ca} + \genE{a,ab,bc,c}.
	\end{split}\]
	\Cref{eq:b} implies that $ab \setminus(a\cup bc) \neq \emptyset$ and $bc \setminus(ab\cup c) \neq \emptyset$, so the corresponding terms are zero.
	So for $\omega$ not being a boundary, it is necessary that $a \subset ab \cup ca$ or $c \subset bc \cup ca$.
	By symmetry, we may assume that
	\begin{equation}\label{eq:a}
	a \subset ab \cup ca
	\end{equation}
	
	Note that $ab$ and $c$ are coprime, so by \cref{lem:paar} there exists a generator $\abSc \in G(\af), \abSc \neq ab, c$ with $\abSc \subset ac \cup c$.
	We claim that \eqref{eq:b} and \eqref{eq:a} imply that $\abSc$ is different from $a,b,c,ab,bc$ and $ca$. Indeed, it is clearly different from $a,b,c$ and $ab$.
	If $\abSc = ca$, then \eqref{eq:a} implies that $a \subseteq ab \cup ca = ab \cup \abSc = ab$, a contradiction to the assumption that $a$ is a minimal generator of $\af$. Hence $\abSc \neq ca$ and
	similarly, \eqref{eq:b} implies that $\abSc \neq bc$.
	In conclusion, $\af$ has at least seven generators.
	
	Next, we note that $a \cap bc = \emptyset$ and $b \cap ca = \emptyset$, so again there has be generators $\bcSa, \caSb \in G(\af)$ with $\bcSa \subset bc \cup a$ and $\caSb \subset ca \cup b$.
	It remains to show that not both of these generators can be equal to some of the generators we already have.
	This is done by arguments very similar to the ones already used, but as this results in a rather extensive and repetitive case distinction, we omit the details.

	\newcommand{\afp}{\af^p}
	Finally, we show that $n \geq 5$. Again, we assume for a contradiction that $n \leq 4$.
	This is obviously not invariant under polarization, so we do not assume that $\af$ is squarefree.
	Instead we denote the polarization of $\af$ by $\afp$.
	
	As before we start by showing that we only need to consider ternary Massey products.
	Consider a nonzero Massey product $\mu_r(a_1, \dotsc, a_r)$ for some $r \geq 3$ and $a_1, \dotsc, a_r \in H_*(K_{S/\af})$.
	The homological degree of the product is $\mu_r(a_1, \dotsc, a_r) = \sum_i |a_i| + r - 2 \geq 2(r-1)$. Hilbert's syzygy theorem implies that $2(r-1) \leq n$, so if $n \leq 4$ then $r = 3$.
	We apply the considerations from above to $\afp$. In particular, $\afp$ has generators $a,b,c,ba,bc,ca$ such that $a,b$ and $c$ are pairwise coprime, and which satisfy \eqref{eq:b} and \eqref{eq:a}.
	But these equations imply that the generators of $\af$ corresponding to $a$ and $b$ cannot be pure powers, so they involve at least $2$ variables each. As both are coprime with $c$, we conclude that there are at least $2+2+1 = 5$ variables.
\end{proof}

The counterexample to \cref{thm:wrong} was found using similar techniques as in the preceding proof.
In particular, we used a computer to compute the boundaries of various element of the Taylor resolution and to extract necessary combinatorial conditions from this.
To check whether the conditions are sufficient to ensure that the Massey product is nonzero, we considered sets $a,b, \dotsc$ \enquote{as generic as possible} with respect to the given constraints and computed the Massey product in $T_\bullet \otimes \kk$.

A rather short computer search yielded several examples of choices of the generators such that the Massey product is non-zero.
Note that the examples found this way are guaranteed to have trivial products of elements of homological degree $1$ and to have $\mu_3(\cls{\genE{a}},\cls{\genE{b}},\cls{\genE{c}}) \neq 0$.
However, one has to check separately that there are no other non-zero products.
Finally, the example of \cref{thm:main} was obtained from such a computer-generated example by the deletion of some variables and by de-polarizing.
Note that in our example it happens that $\bcSa = \caSb$.

\section{Remarks and Questions}\label{sec:remark}
In this last section, we collect  some remarks and an open question.

\subsection{The gap in the proof of \texorpdfstring{\cref{thm:wrong}}{Claim \ref{thm:wrong}}}
Let us briefly discuss what seems to be the reason for the failure of \cref{thm:wrong}.
First, this result was stated in \cite[Theorem 7.1]{Jo06} under the additional assumption that $R = S/\af$ satisfies a certain property \textbf{(P)}.
In that article, it was conjectured that every monomial ring has this property, and that conjecture was then confirmed in \cite{BJ07}, leading to the unconditional statement of \cref{thm:wrong} in \cite[Theorem 5.1]{BJ07}.

The problem with this proof seems to lie in \cite[Theorem 7.1]{Jo06}.
Its proof goes by applying discrete Morse theory to the Taylor resolution of $S/\af$. 
Here, a special type of Morse matching on the Taylor complex is used, a so--called \emph{standard matching}.
We refer the reader to \cite[Definition 3.1]{Jo06} for the precise definition.
A standard matching is compatible with the multiplicative structure of the Taylor complex in a certain way.
This compatibility is crucial for the study of the multiplicative structure on $\Tor_*^S(S/\af,\kk)$.

It is stated in \cite[p. 268]{Jo06} that such a standard matching always exists, but this is not true.
For example, it is not difficult to see from the definition that the ideal
\[ \af = (x_1^2, x_1x_2, x_2x_3, x_3x_4, x_4^2) \subset \kk[x_1,\dotsc, x_4] = S \]
does not allow a standard matching.
In fact, this example is taken from \cite[Example 2.2]{AvraObstr} where it is given as an example of an ideal whose minimal free resolution does not allow a DGA structure.
A standard matching does not induce a DGA structure in general, but it is related.
Therefore, is seems plausible to look for ideals not admitting a standard matching among the known examples of ideals whose minimal free resolution does not allow a DGA structure.
As this example does not satisfy $(B_2)$, it does not directly yield a counterexample to \cref{thm:wrong}.

Recently, de Stefani found two monomial ideals whose product is not Golod \cite{deS}.
This yields a counterexample to another result of \cite{BJ07}, namely Theorem 5.5 in \emph{loc.cit.}, which states that the so-called \emph{strong gcd-condition} implies Golodness.
Here the actual error is in \cite{Jo06} as well, but is it not the same as the one behind the failure of \Cref{thm:wrong}.

\subsection{A general bound for the Massey products}
We wonder whether \cref{thm:reg} holds more generally:
\begin{question}
	Is the assumption \enquote{monomial} in \cref{thm:reg} really necessary?
	More precisely, let $\af \subset S$ be a homogeneous ideal, let $R := S/\af$ and let $r := \max(2, \reg S/\af - 2)$.
	If $K_R$ satisfies $(B_{r})$, does it follow that $R$ is Golod?
\end{question}
Note that to answer this question, it would be enough to prove \cref{thm:zero} for general graded rings.
For completeness, we also note the following criteria for the Golod property, which look similar to \cref{thm:reg} but are actually rather straightforward:
\begin{proposition}
	Let $\af \subset S = \kk[x_1, \dots, x_n]$ be a homogeneous ideal and let $R = S/\af$.
	\begin{enumerate}
		\item If $K_R$ satisfies $(B_{\lfloor p/2\rfloor + 1})$ for $p = \pdim R$, then it is Golod.
		\item If $\af$ is a squarefree monomial ideal which contains no variable and $R$ satisfies $(B_{\lfloor n/2\rfloor})$, then $R$ is Golod.
	\end{enumerate}
\end{proposition}
\begin{proof}
	The first claim follows easily by considering the homological degree of a Massey product.
	
	Under the assumptions of the second claim, every nonzero Koszul cycle has at least two nonzero components in its multidegree.
	So any Massey product of more than $\lfloor n/2\rfloor$ factors is zero for multidegree reasons.
\end{proof}

\subsection{Special classes of monomial ideals}
We close this article with a collection of several classes of monomial ideals for which zero products imply zero Massey products. 

\begin{theorem}\label{thm:classes}
	Let $\af \subset S$ be a monomial ideal and let $R := S/\af$.
	Then $R$ is Golod, if the product in the Koszul homology of $R$ is trivial and one of the following conditions holds:
	\begin{enumerate}
		\item $\af$ is generated by monomials of degree $2$,
		\item $\af$ is generic (in the sense of \cite[Definition 1.1]{MSY}), 
		\item $\af$ has at most seven generators,
		\item $\dim S \leq 4$,
		\item $\af$ is squarefree and $\dim S \leq 8$,
		\item $\reg S/\af \leq 4$, or
		\item $\af$ is the Stanley-Reisner ideal of a simplicial complex of dimension at most $3$.
	\end{enumerate}
\end{theorem}
\begin{proof}
	\begin{asparaenum}
		\item[(1)] This follows from the characterization of the Golod property among \emph{flag} simplicial complexes, \cite[Theorem 6.4]{BJ07}.
		The proof of this result given in \cite{BJ07} depends on \cref{thm:wrong}, but there is an independent proof by Grbic, Panov, Theriault and Wu \cite[Theorem 4.6]{GPTW12}. 
		\item[(2)]
		It follows from the description of the minimal free resolutions via the Scarf complex (cf. \cite[Theorem 1.5 (f)]{MSY}) that for each multidegree $\ab \in \NN^n$, there exists at most one homological degree $i \in \NN$ such that $\Tor^S_i(S/\af,\kk)_\ab \neq 0$.
		Such an $i$ exists if and only if there is a unique subset $I(\ab) \subset G(\af)$ of the generators of $\af$ such that $\ab = \DEG \lcm\set{m \in I(\ab)}$, and in this case it holds that $i = \#I(\ab)$.
		
		Assume that there exists a nonzero Massey product $\omega \in \mu_r(\alpha_1,  \dotsc, \alpha_r)$ for $r \geq 3$ and homogeneous elements $\alpha_1,  \dotsc, \alpha_r \in \Tor^S_*(S/\af,\kk)_\ab$.
		From $\DEG \omega = \sum_{i=1}^r \DEG \alpha_i$ and the uniqueness of $I(\omega)$, it follows that $I(\DEG \omega) = \bigcup_{i=1}^r I(\DEG \alpha_i)$.
		Considering the homological degrees yield the following:
		\[\begin{split}
		\sum_{i=1}^r |\alpha_i| = \sum_{i=1}^r \#I(\DEG \alpha_i) \geq \#\left(\bigcup_{i=1}^r I(\DEG \alpha_i)\right) 
		= \# I(\omega) \\= |\omega| = \sum_{i=1}^r |\alpha_i| + r - 2 
		\end{split}\]
		It follows that $r \leq 2$, so the only Massey products which can possibly be nonzero are the binary Massey products.
		\item[{(3),(4)}] This is \cref{thm:min}.
		\item[(5)] Any counterexample to \cref{thm:wrong} has at least three generators which are pairwise coprime.
		Further, by \cref{thm:zero} we may also assume that these generators have degree at least three.
		But three squarefree pairwise coprime monomials of degree at least three can only exist if the ambient ring has at least nine variables.
		\item[(6)] This is immediate from \cref{thm:reg}.
		\item[(7)] This is a special case of part 6).
	\end{asparaenum}
\end{proof}

\begin{remark}	
	The assertion of part (2) of the preceding theorem follows also from the statement of \cite[Prop. 2.5]{chara}.
	However, that result builds on \cite[Corollary 3.6]{BPS}, which claims that the minimal free resolution of a generic monomial ideals admits the structure of a DGA.
	As I recently found a counterexample to the latter (cf. \cite[Theorem 5.1]{Golod3}), I included a full proof of part (2) of \Cref{thm:classes}.
\end{remark}

\begin{remark}
	The parts (3) and (4) of \Cref{thm:classes} both fail for non-monomial ideal.
	Indeed, Roos \cite{RoosExample} recently found an example of a non-monomial ideal $I \subset S$ in four variables with six generators, such that $S/I$ is not Golod but has only zero products in its Koszul homology.
\end{remark}

\section*{Acknowledgment}
The author thanks Sean Tilson and Alessandro de Stefani for inspiring discussions.
I also thank Jan-Erik Roos for his idea to simplify the main example, and for sharing an early version of \cite{RoosExample}.
Finally, I thank the referee and Luchezar Avramov for various helpful comments and suggestions.

\printbibliography

\end{document}